\documentclass[11pt,a4]{amsart}
\usepackage{amsfonts,amsmath,amssymb,bbm}
\usepackage{verbatim}
 \usepackage{setspace}

\makeatletter
\renewcommand{\section}{\@startsection{section}{1}{0pt}{20pt}{6pt}{\large\bfseries}}
\makeatother

\usepackage{enumerate}
\usepackage{pdfsync}
\usepackage{color}

\numberwithin{equation}{section}

\theoremstyle{plain}
  \newtheorem{theorem}{Theorem}[section]
  
  \newtheorem{lemma}[theorem]{Lemma}
   \newtheorem{corollary}[theorem]{Corollary}

\theoremstyle{definition}

\theoremstyle{remark}
\newtheorem{remark}[theorem]{Remark}

\newcommand{\e}{{\mathbf{e}}}
\newcommand{\E}{\mathbb{E}}
\renewcommand{\P}{{\mathbb{P}}}

\newcommand{\R}{\mathbb{R}}
\newcommand{\C}{\mathbb{C}}
\newcommand{\N}{\mathbb N}

\newcommand{\lbb}{\left [}
\newcommand{\rbb}{\right ]}
\newcommand{\labs}{\left |}
\newcommand{\rabs}{\right |}

\newcommand{\A}{\mathcal{A}}
\newcommand{\G}{\mathcal{G}}

\renewcommand{\P}{\mathcal{P}}

\newcommand{\Id}[1]{{{\mathbb{I}}}_{\{#1\}}}
\newcommand{\M}{{\rm{M}}}

\title{On Doney's striking factorization of the arc-sine law}%\protect\thanksref{T1}}
\author{L.~Alili}
\address{Department of Statistics, The University of Warwick, CV4 7AL, Coventry, UK}
\email{L.Alili@warwick.ac.uk}

\author{C.~Bartholm\'e}
\address{Universit\'e de Luxembourg}
\email{carine.bartholme@gmail.com}

\author{L. Chaumont}
\address{LAREMA UMR CNRS 6093, Universit\'e d’Angers, 2, Bd Lavoisier,
Angers Cedex 01, 49045, France.}
\email{loic.chaumont@univ-angers.fr}

\author{P.~Patie}
\address{School of Operations Research and Information Engineering, Cornell University, Ithaca, NY 14853.}
\email{pp396@cornell.edu}

\author{M. Savov}
\address{Institute of Mathematics and Informatics, Bulgarian Academy of Sciences,  "Akad. Georgi Bonchev" bl. 8, Sofia 1113, Bulgaria.}
\email{mladensavov@math.bas.bg}

\author{S. Vakeroudis}
\address{Department of Mathematics, Track: Statistics and Actuarial-Financial Mathematics, University of the
Aegean, Vourliotis Building, Office: Y5, 83200 Karlovasi, Samos, Greece.}
\email{stavros.vakeroudis@gmail.com}

\begin{document}

\maketitle

\begin{abstract}
In \cite{D87}, R. Doney identifies a striking factorization of the arc-sine law in terms of the suprema of two independent stable processes of the same index
by an elegant random walks approximation. In this paper, we provide an alternative proof and a generalization of this factorization based on the  theory recently developed  for  the exponential functional of  L\'evy processes. As a by-product, we  provide some interesting distributional properties for these variables and also some new examples of the factorization of the arc-sine law. % which is known to be intimately connected to the suprema of stable processes. % Several interesting implications of this factorization will also be discussed.
\end{abstract}

\section{Introduction}
Let $\M_{\rho}=\sup_{0\leq t \leq 1} X_{t}$ and $\widehat{\M}_{\rho}=\sup_{0\leq t \leq 1} -\widetilde{X}_{t}$ where $X= (X_{t})_{t\geq0}$ and $ \widetilde{X}=(\widetilde{X}_{t})_{t\geq 0}$ are two independent copies of a stable process of index $\alpha \in (0, 2)$ and positivity parameter $\rho \in (0,1)$.  Doney \cite[Theorem 3]{D87} proved the following factorization of the  arc-sine random variable $\A_{\rho}$ of parameter $\rho$
	\begin{equation} \label{id-Doney}
	        \frac{\M_{\rho}^{\alpha}}{\M_{\rho}^{\alpha}+\widehat{\M}_{\rho}^{\alpha}} \stackrel{\left(d\right)}{=} \A_{\rho}
	\end{equation}
where $\stackrel{\left(d\right)}{=}$ stands for the identity in distribution, and the law of $\A_{\rho} $ is absolutely continuous with a density given by
\begin{equation*}
\frac{\sin(\pi \rho)}{\pi} x^{\rho-1} (1-x)^{-\rho}, \ x \in \left( 0,1 \right).
\end{equation*}
The distributional identity \eqref{id-Doney} is remarkable because the law of the supremum of a stable process is usually a very complicated object whereas the arc-sine law has  a simple distribution. In recent years, the law of $\M_{\rho}$ has been the interest of many researchers, see e.g.~\cite{DS, K11,KP13,P12} where we can find series or Mellin-Barnes integral representations for the density of the supremum of a stable process valid for some set of parameters $\left(\alpha, \rho \right)$.
 To retrieve the identity \eqref{id-Doney} from one of these representations does not seem straightforward.
 We mention that Doney resorts to a limiting procedure to derive the factorization \eqref{id-Doney} of the arc-sine law. More specifically, his proof stems on a  combination of  an identity for each path of a random walk in the domain of attraction of a stable law with the arc-sine theorem which can be found in Spitzer \cite{Spitzer56}. We also mention that the arc-sine law appears surprisingly  in different contexts in probability theory and in particular in the study of functionals of Brownian motion, see e.g.~\cite{PL, PY2, YC,VY}.

The aim of this work is to provide an alternative proof and offer a  generalization of Doney's  factorization of the arc-sine law. The first key step relies on the well-known fact by now that, through the so-called Lamperti mapping, one can relate the law of the supremum of a stable process to the one of the exponential functional of a specific L\'evy process, namely the Lamperti-stable process introduced by  Caballero and Chaumont \cite{CC}.  It is then natural to wonder whether there are  other factorizations of the arc-sine law given in terms of exponential functionals of more general  L\'evy processes. This will be achieved by resorting to the thorough  study on the functional equation satisfied by the Mellin transform  of the exponential functional of L\'evy processes carried out in Patie and Savov \cite{Patie-Savov-Bern}.

Besides proving these identities in a more general framework, the problem of identifying a factorization of the exponential functionals as a simple distribution is interesting on its own since we shall show, on the way, that the law of the ratio of independent exponential functionals of some L\'evy processes is the Pareto's one which is known to  belong to some remarkable sets of probability laws. This new fact  is also relevant as  the exponential functional of L\'evy processes  has attracted the attention of many researchers over the last two decades. The law of this random variable plays an important role in the study of self-similar processes, fragmentation and branching processes and is related to other theoretical problems as for example the moments problem and spectral theory of some non-self-adjoint semigroups, see \cite{Patie-Savov-AMS}. Moreover it also plays an important role in more applied domains as for example in mathematical finance for the evaluation of Asian options, in actuarial sciences for random annuities, as well as in domains like astrophysics and biology. We refer to the survey paper \cite{BY05} for a more detailed account on some of the mentioned fields.
The remaining part of the paper is organized as follows. We state our main factorization of the arc-sine law along with some consequences and examples in the next section. The last section is devoted to the proofs.

\section{The  arc-sine law and exponential functional of L\'evy processes} \label{subsec : main results}
Throughout this paper we denote by $\xi=(\xi_{t})_{t\geq0}$ a possibly killed L\'evy process issued from $0$ and defined on the probability space $(\Omega, \mathcal F, \mathbb P)$. It means that $\xi$ is a real-valued stochastic process   having independent and stationary increments and possibly killed at the random time $\e_{q}$, which is independent of $\xi$ and exponentially distributed with parameter $q \geq 0$, where we understand that $\e_{0}=+\infty$. We denote by $\Psi$ its L\'evy-Khintchine exponent, which, for any $z \in i \R$, takes the form
\begin{eqnarray}\label{eq:lap-levy}
\log \E\left[ e^{z \xi_{1}} \right]=\Psi(z) =az + \frac{1}{2} \sigma^2 z^2 + \int_{\R} \left(e^{z y}-1
-zy\Id{|y|<1} \right)\Pi(dy)-q
\end{eqnarray}
where  $a\in \R$, $\sigma \geq 0$ and
$\Pi$ is a Radon measure on $\R$ satisfying the conditions $\int_{\R} (1
\wedge y^2 )\:\Pi(dy) <+ \infty$ and $\Pi(\{0\})=0$. The law of $\xi_1$ is infinitely divisible and the one of $\xi$ is uniquely characterized by the quadruplet $(q,a,\sigma,\Pi)$.
An excellent account on L\'evy processes can be found in the monographs \cite{B96,Dbook, Kyprianou06, Sato99}.
Next, we define the exponential functional associated to the L\'evy process $\xi$ by
\begin{equation} \label{def: exponential functional}
{\rm{I}}_{\Psi} = \int_0^{\infty} e^{\xi_t} dt= \int_0^{\e_{q}} e^{\xi_t} dt.
\end{equation}
%and denote the exponential functional associated with its dual L\'evy process $-\xi$ by
%\begin{equation*}
%\widehat{{\rm{I}}}_{\Psi} = \int_0^{+\infty} e^{-\xi_t}  dt.
%\end{equation*}
%Note that ${\rm{I}}_{\Psi}$ and $\widehat{{\rm{I}}}_{\Psi}$ are well defined for any $\Psi \in \mathcal{N}^{-}$ and $\Psi \in \mathcal{N}^{+}$ respectively. For any $\Psi \in \mathcal{N}^{-}$, we have that $\widehat{\Psi} \in \mathcal{N}^{+}$, with $\widehat{\Psi}$ the characteristic exponent associated with the dual L\'evy process $\widehat{\xi}=- \xi$, and we also obviously have ${\rm{I}}_{\Psi}=\widehat{{\rm{I}}}_{\widehat{\Psi}}$.
The variable ${\rm{I}}_{\Psi}$ is well defined if either $\Psi(0)=-q<0$ or $\lim_{t\rightarrow+\infty}\xi_{t}= - \infty$ a.s. This last condition is equivalent to Erickson's integral tests involving the L\'evy measure $\Pi$ and the drift $a$, see Bertoin and Yor \cite[Theorem 1]{BY05}. With this remark in mind, we denote by $\mathcal{N}$ the set of L\'evy Khintchine exponents of the form \eqref{eq:lap-levy} for which the exponential functional ${\rm{I}}_{\Psi}$ is well defined, i.e.
\begin{equation}
\mathcal{N} = \left\{ \Psi \mbox{ of the form \eqref{eq:lap-levy}} ; \:  \Psi(0)<0 \mbox{ or } \lim_{t\rightarrow+\infty}\xi_{t}= - \infty \mbox{ a.s.} \right\}.
\end{equation}
Note that $\mathcal{N}$ is a subspace of the negative of continuous negative-definite functions, as defined in \cite{J01}.
Next, it is well-known that  $\Psi\in \mathcal{N}$ admits an analytical extension (still denoted by $\Psi$) to the strip $\mathbb{C}_{(0,\beta)}=\{z\in \C; \: 0<\Re(z)<\beta\}$ with $ 0<\beta$ if and only if $\labs\E\lbb e^{z\xi_1}\rbb\rabs< \infty$ for all $z\in \mathbb{C}_{(0,\beta)}$. Note that the existence of exponential moments for all $z\in \mathbb{C}_{(0,\beta)}$ is equivalent to
\begin{equation} \label{eq: exp mom Sato cond}
\int_{ y >1} e^{u y} \Pi(dy) < \infty \mbox{ for all } u \in (0,\beta).
\end{equation}
%where the equality serves to set a notation.
Under this condition, the restriction of $\Psi$ on the real interval $(0,\beta)$ is convex and  the condition
 $\lim_{t\rightarrow+\infty}\xi_{t}= - \infty$ a.s.~is equivalent to $\Psi'(0^{+})<0$, see e.g.~\cite[Theorem 1 and Remark p.193]{BY05}. We then define for any $\beta>0$,
 \[ \mathcal{N}_{\beta}=\left\{ \Psi \in \mathcal{N}; \: \eqref{eq: exp mom Sato cond} \textrm{ holds}  \right\}. \]
 %\begin{equation*}
%\mathcal{N}_{\beta}= \left\{\Psi \in \mathcal{N}; \Psi \mbox{ is analytical in }   \mathbb{C}_{(0,\beta)}, \Psi'(0^{+})<0 \right\}.
%\end{equation*}
Next, for any  $\Psi \in \mathcal{N}_{\beta}$, let us denote \[ \rho= \sup\{u \in (0,\beta); \: \Psi(u)=0\}\] with the usual convention  that $\sup \{\emptyset \}=+ \infty$ and,
 introducing the notation
\begin{equation*}
%\overline{\Pi}_{-}\left(y\right)=\int_{\left[-\infty,y\right]} \Pi_{-}(dr)  \mbox{ and }
\overline{\Pi}_{+}\left(y\right)=\int_{y}^{\infty} \Pi(dr) \mathbb{I}_{\{y>0\}},
\end{equation*}
 we  define
\begin{equation*}
\mathcal{N}_{\beta}(\rho) = \left\{ \Psi \in \mathcal{N}_{\beta}; \rho<\infty, y \mapsto e^{\beta y} \overline{\Pi}_{+}\left(y\right)\textrm{ is non-increasing},  \infty<\lim_{u\uparrow 0} u \Psi(u+\beta) \leq 0 \right\}.
\end{equation*}
We point out that if $\Psi \in \mathcal{N}_{\beta}(\rho)$ and $\lim_{u\uparrow 0} u \Psi(u+\beta)$ exists then necessarily $\lim_{u\uparrow 0} u \Psi(u+\beta) \leq 0$ as, by definition $\rho<\beta$, and $\Psi$ is convex increasing on $(\rho, \beta)$. Note also that for any $\Psi \in \mathcal N$ with $\overline{\Pi}_{+} \equiv 0$,  we always have $0<\rho <\infty$ and thus $\Psi \in \mathcal{N}_{\beta}(\rho)$ for all $\beta >\rho$. We also point out if $|\Psi(\beta)|<\infty$, that is $\Psi$ extends continuously to the line $\beta+i\R$, then plainly $\lim_{u\uparrow 0} u \Psi(u+\beta)=0$. We are now ready to state our main result.

\begin{theorem} \label{mainthm}
Assume that $ \Psi \in \mathcal{N}_{1}(\rho)$ with  $0<\rho<1$, then  $\widehat{\Psi}_1(z)=\Psi_1(-z) \in \mathcal{N}_1$ with $\widehat{\Psi}_1(1-\rho)=0$,  where
\begin{equation*}
\Psi_{1}(z)=\frac{z}{z+1}{\Psi}\left(z+1\right), \: z \in i\R,
\end{equation*}
and,
%Consequently, we have the following factorizations of the generalized arc-sine law
\begin{equation} \label{thm-Arc}
\frac{{\rm{I}}_{{\widehat{\Psi}_{1}}}}{{\rm{I}}_{{\widehat{\Psi}_{1}}}+{\rm{I}}_{{\Psi}}}  \stackrel{\left(d\right)}{=} \A_{\rho} \mbox{ and }  \frac{{\rm{I}}_{{\Psi}}}{{\rm{I}}_{{\widehat{\Psi}_{1}}}+{\rm{I}}_{{\Psi}}}  \stackrel{\left(d\right)}{=} \A_{1-\rho}
\end{equation}
where the variables  ${\rm{I}}_{{\Psi}}$ and ${\rm{I}}_{{\widehat{\Psi}_{1}}}$ are taken independent.
\end{theorem}
We proceed by providing some consequences of this main result. We first derive some interesting distributional properties for the ratio of independent exponential functionals. To this end, we recall that a positive random variable  is  hyperbolically completely monotone if its law is absolutely continuous with a probability density $f$ on $(0, \infty)$  which is such that the function $h$ defined  on $(0, \infty)$  by
\begin{equation}
h(w)=f\left(uv\right)f\left(u/v\right), \ \mbox{ with } w=v+v^{-1},
\end{equation}
is, for each fixed $u>0$, completely monotone, i.e.~$\left(-1\right)^{n}\frac{d^n}{dw^n}h(w)\geq0$ on $(0, \infty)$ for all integers $n\geq0$.
This  remarkable set of random variables was introduced by Bondesson and  in \cite[Theorem 2]{Bond}, he shows  that it is a subset of the class of generalized gamma convolution. We recall that  a positive random variable belongs to this latter class if it is self-decomposable, and hence infinitely divisible, such that its L\'evy measure $\Pi$, concentrated on $\R^+$, is such that $\int_{0}^{\infty}(1 \wedge y) \Pi(dy)<\infty$ and $\Pi(dy)=\frac{k(y)}{y}dy$  where  $k$ is completely monotone. We also say that a positive random variable ${\rm{I}}$ is  multiplicative infinitely divisible if $\log I$ is infinitely divisible. It turns out that under some conditions the random variables  ${\rm{I}}_{{\Psi}}$   is multiplicative infinitely divisible, see
\cite[Theorem 1.5]{Alili-Jadidi-14} when $\Psi$ is a Bernstein function   and \cite[Theorem 4.7]{Patie-Savov-Bern} in the general case.

\begin{corollary} \label{cor:par}
With the notation and assumptions of Theorem \ref{mainthm}, the random variables $\frac{{\rm{I}}_{{\Psi}}}{{\rm{I}}_{{\widehat{\Psi}_{1}}}}$ and $\frac{{\rm{I}}_{{\widehat{\Psi}_{1}}}}{{\rm{I}}_{{\Psi}}}$ are
hyperbolically completely monotone and multiplicative infinitely divisible. Moreover, when $\rho=\frac12$, then $\frac{{\rm{I}}_{{\Psi}}}{{\rm{I}}_{{\widehat{\Psi}_{1}}}}$ is self-reciprocal, i.e.~it has the same law than $\frac{{\rm{I}}_{{\widehat{\Psi}_{1}}}}{{\rm{I}}_{{\Psi}}}$, and it has the law of $\mathcal C^2$ where $\mathcal C$ is a standard Cauchy variable.
\end{corollary}

Another  consequence of Theorem \ref{mainthm} is  the following.
\begin{corollary}\label{cor:sr}
Doney's identity \eqref{id-Doney} holds.
\end{corollary}
We close this section by describing another example illustrating our main factorization of the arc-sine law with some classical variables and refer the interested reader to the thesis \cite{Carine} for the description of additional examples.  Let us consider first  $S(\alpha)$  a positive $\alpha$-stable variable, with $0<\alpha<1$, and denote by $S_{\gamma}^{-\alpha}(\alpha)$ its  $\gamma$-length-biased, $\gamma>0$, that is for any bounded measurable function $g$ on $\R^+$, one has
\begin{equation*}
\E\left[ g(S_{\gamma}^{-\alpha}(\alpha))\right] = \frac{\E\left[S^{-\alpha \gamma}(\alpha) g(S^{-\alpha}(\alpha))\right]}{ \E\left[S^{-\alpha\gamma}(\alpha)\right]}
\end{equation*}
where we recall that $\E\left[S^{-\alpha\gamma}(\alpha)\right]<\infty$, see e.g.~\cite[Section 3(3)]{P11}. We also denote by $\G_{a}$ a gamma variable of parameter $a>0$.
 \begin{corollary}\label{cor:s2}
 Let  $0<\alpha,\rho<1$, then we have the following factorization of the arc-sine law
\begin{equation*} %\label{ex:ID3}
\frac{ \G_{\alpha (1-\rho)}^{-\alpha}}{\G_{1-\rho}^{-1}S_{\rho}^{-\alpha}(\alpha) + \G_{\alpha (1-\rho)}^{-\alpha}} \stackrel{\left(d\right)}{=} \A_{\rho}
\end{equation*}
where the three variables $\G_{\alpha (1-\rho)}, S_{\rho}(\alpha)$ and $\G_{1-\rho}$ are taken independent.
\end{corollary}
\section{Proofs} \label{sec: proofs}
The proof of Theorem \ref{mainthm} is split into several intermediate results which might be of independent interests.
 First, let $(\mathcal{T}_{\beta})_{\beta \in \R}$ be the group of transformations defined, for a function $f$  on the complex plane, by
\begin{equation} \label{eq:simpletransform}
\mathcal{T}_{\beta} f(z)=\frac{z}{z+\beta}f\left(z+\beta\right).
\end{equation}
In what follows, which is a slight extension of \cite[Proposition 2.1]{PS12}, we show that under mild conditions, this family of transformations  enables to identify an invariance property of the subset of L\'evy-Khintchine exponents. Note that this lemma contains the first claim of Theorem \ref{mainthm}.
\begin{lemma} \label{lemma-lintrans}
Let $\beta_+>0$ and $\Psi$ be of the form \eqref{eq:lap-levy}  such that for any $\beta \in (0,\beta_+), |\Psi(\beta)|<\infty $. Then, for any $\beta \in (0,\beta_+]$ such that
\[  y\mapsto e^{\beta y} \overline{\Pi}_{+}\left(y\right) \textrm{ is non-increasing on } \R^+ \textrm{ and } \infty<q_{\beta}=\lim_{u\to 0}\mathcal{T}_{\beta}\Psi\left(u\right)\leq 0,\]
 we have that $\mathcal{T}_{\beta}\Psi$ is also   of the form \eqref{eq:lap-levy}.  More specifically, its killing  rate is $q_{\beta}$,   its Gaussian coefficient is $\sigma$ and  its L\'evy measure takes the form
\begin{equation} \label{LM-GenCKP1}
  \Pi_{\beta}\left(dy\right)=  e^{\beta y}\left( \Pi\left(dy\right) +\beta dy\left(\overline{\Pi}_{-}\left(y\right) -\overline{\Pi}_{+}\left(y\right)+q\mathbb{I}_{\{y<0\}} \right)\right), \: y \in \R.
 \end{equation}
 Finally, if, in addition, $\Psi \in \mathcal{N}_\beta(\rho)$ with $\beta \in (\rho,\beta_+]$, then $\widehat{\Psi}_{\beta} \in \mathcal{N}_\beta $ with $\widehat{\Psi}_1(\beta-\rho)=0$.
\end{lemma}
\begin{remark}
  Note that when $|\Psi(\beta)|<\infty $ then immediately  $q_{\beta}=0$. Moreover, the situation $|\Psi(\beta_+)|=\infty $ is allowed  if $0$ is a removable singularity for  $\mathcal{T}_{\beta_+}\Psi$ with $q_{\beta_+}=\mathcal{T}_{\beta_+}\Psi(0)\leq 0$.
  \end{remark}
\begin{proof}
For any $\beta \in (0,\beta_+)$, since in this case plainly $q_{\beta}=0$,  the claim is given in \cite[Proposition 2.1]{PS12} and thus it remains to prove it only for $\beta=\beta_+$.  Note also that the expression of the characteristics of $\mathcal{T}_{\beta_+}\Psi$ follows from this aforementioned result and we now show that it is indeed a characteristic exponent of a L\'evy process. To this end,  we recall a few properties of the set of all negative definite functions $N(\R)$ and the set of all continuous negative definite functions denoted by $CN(\R)$ and refer to the monograph \cite{J01} for an excellent account  on these sets of functions.
A function $f:\R \longrightarrow \C$ is an element in $N(\R)$ if and only if the following conditions are fulfilled
$f(0) \geq 0$, $f(z)= \overline{ f(-z)}$, and for any $k \in \N$ and any choice of values $z^1, \cdots, z^k \in \R$ and complex numbers $c_1, \cdots, c_k$
\begin{equation}\label{eq:NDF}
\sum^{k}_{j=1} c_j=0  \mbox{ implies that } \sum^{k}_{j,l=1} f(z^j-z^l)c_j  \overline{c_l} \leq 0.
\end{equation}
It is easy to verify that $-\Psi(-z) \in CN(\R)$, $z \in i\R$ and it is also well-known that any element of $CN(\R)$ can be written as the negative of a characteristic function of a L\'evy process. Now, we have, for any $\beta \in (0,\beta_+)$,  $\mathcal{T}_{\beta} \Psi \left(z\right)$, $z \in i\R$, is the characteristic exponent of a conservative ($q_{\beta}=0$) L\'evy process. Denote
\begin{equation}
\mathcal{T}_{\beta^{+}} \Psi =  \lim_{\beta \rightarrow \beta^{+}} \mathcal{T}_{\beta} \Psi.
\end{equation}
Note that $\mathcal{T}_{\beta^{+}} \Psi (0):=\lim_{u\uparrow 0}\mathcal{T}_{\beta}\Psi\left(u\right)= q_{\beta}$ which is a non-negative constant by assumption. Then, let
  \begin{equation*}
  \Phi(z)=\left\{
  \begin{array}{ll}
  - \mathcal{T}_{\beta^{+}} \Psi(z)& \mbox{ if } z\neq0 \\
  0 & \mbox{ if } z=0.
  \end{array}
  \right.
    \end{equation*}
 Then $  \Phi$ is an element of  $N(\R)$ since we know from \cite[Lemma 3.6.7, p.123]{J01} that  the set $N(\R)$ is a convex cone which is closed under pointwise convergence. Let further
   \begin{equation*}
  \widetilde{\Phi}(z)=\left\{
  \begin{array}{ll}
   -q_{\beta}+\Phi(z)& \mbox{ if } z\neq0 \\
  -q_{\beta} & \mbox{ if } z=0,
  \end{array}
  \right.
    \end{equation*}
    where we recall that $q_{\beta}\leq0$.
For any $k \in \N$ and any choice of values $z^1, \cdots, z^k \in \R$ and complex numbers $c_1, \cdots, c_k$ with $\sum^{k}_{j=1} c_j=0$, since $  \Phi \in N(\R)$ and $-q_{\beta} \in N(\R)$, we have
\begin{align*} \sum^{k}_{j,l=1}   \widetilde{\Phi}(z^j-z^l)c_j  \overline{c_l}  & =  \sum^{k}_{j,l=1; z^j \neq z^l}   \Phi(z^j-z^l)c_j  \overline{c_l} - q_{\beta} \sum^{k}_{j,l=1; z^j \neq z^l}  c_j
\overline{c_l} - q_{\beta} \sum^{k}_{j,l=1;z^j=z^l}  c_j
\overline{c_l} \\
& =  \sum^{k}_{j,l=1}   \Phi(z^j-z^l)c_j  \overline{c_l} - q_{\beta} \sum^{k}_{j,l=1}  c_j
\overline{c_l} \leq \sum^{k}_{j,l=1}   \Phi(z^j-z^l)c_j  \overline{c_l} \leq 0.
\end{align*}
Hence $\widetilde{\Phi} \in N(\R)$ and by continuity,  $\mathcal{T}_{\beta^{+}} \Psi$  is the characteristic exponent of a possibly killed L\'evy process. Next, let us assume that, in addition, $\Psi \in \mathcal{N}_\beta(\rho)$ with $\beta \in (\rho,\beta_+]$, then writing $\widehat{\Psi}_{\beta}= \widehat{\mathcal T_{\beta}\Psi}$, $\widehat{\Psi}_{\beta}$ is the characteristic exponent of the dual L\'evy process associated to ${\mathcal{T}_{\beta}\Psi}$ and, as  $\widehat{\Psi}_{\beta}(z)=\frac{z}{z-\beta}\Psi(-z+\beta)$, we have that  $\widehat{\Psi}_{\beta}(z)$ is analytical on the strip $\mathbb{C}_{(0,\beta)}$ and  $ \widehat{\Psi}_{\beta}(\beta-\rho)=0$. Moreover, since for $u \in  (0,\beta)$, $\widehat{\Psi}_{\beta}'(u)=-\frac{u}{u-\beta}\Psi'(-u+\beta)-\frac{\beta}{(u-\beta)^2}\Psi(-u+\beta)$, we get, if $\beta \in (\rho,\beta_+)$, that  $\widehat{\Psi}_{\beta}'(0)=-\frac{\Psi(\beta)}{\beta}<0$ as $\beta>\rho$, either $\Psi(0)<0$ or $\Psi'(0^+)<0$ and $\Psi$ is convex on $(0,\beta)$. Hence, in this case,   $\widehat{\Psi}_{\beta} \in \mathcal{N}_\beta(\beta-\rho) $. Finally, if $q_{\beta_+}<0$ then clearly $\widehat{\Psi}_{\beta_+} \in \mathcal{N}_{\beta_+}(\beta-\rho) $ whereas if $q_{\beta_+}=0$  we complete the proof by recalling  that $ \widehat{\Psi}_{\beta_+}(\beta_+-\rho)=0$ with $\beta_+-\rho>0$ and the convexity of $\widehat{\Psi}_{\beta}$.
\end{proof}
Next, we denote by $\mathcal{M}_{{\rm{I}}}$  the Mellin transform of  a random variable ${\rm{I}}$, that  is, for $z\in \C$,
\begin{equation*}
  \mathcal{M}_{{\rm{I}}}(z)=\E[{\rm{I}}^{z-1}]
\end{equation*}
and, recall that $\mathcal{M}_{{\rm{I}}}$ is (on its domain of definition) a positive-definite function. We proceed by recalling a few basic facts about the generalized Pareto random variable, which we denote by $\P_{a,b}, a,b>0$, that will be useful in the sequel of the proof. It can be defined via the identity
\begin{equation} \label{pareto-gamma}
\P_{a,b}  \stackrel{\left(d\right)}{=} \frac{\G_ {b}}{\G_ {a}}
\end{equation}
where  the two variables, which we recall are gamma variables of parameter $b$ and $a$ respectively,  are considered independent. We recall that the law of $\G_{a}$ is absolutely continuous with the following  density
\begin{equation*}
\frac{1}{\Gamma(a)} x^{a-1}e^{-x}, \: x>0. %\mathbb{I}_{\left\{x>0\right\}},
\end{equation*}
$\P_{a,b}$ is also a positive variable whose law is absolutely continuous with a density given by
\begin{equation*}% \label{Pareto-density}
\frac{1}{B(a,b)} x^{b-1} \left(1+x\right)^{-a-b}, \: x>0, %\mathbb{I}_{\left\{x>0\right\}},
\end{equation*}
where $B(a,b)= \frac{\Gamma(a) \Gamma(b)}{\Gamma(a+b)}$ is the Beta function. The Mellin transform of $\P_{a,b}$ is given by
\begin{equation} \label{Pareto-kth moment}
\E\left[ \P_{a,b}^{z}\right]=\frac{\Gamma(a-z) \Gamma(b+z)}{\Gamma(a) \Gamma(b)}, \ -b<\Re(z)<a.
\end{equation}
%and its Fourier transform is given, for $z \in i \R$, by
%\begin{equation} \label{Pareto-Fourier transform}
%\E\left[ e^{z \P_{a,b}}\right]= \frac{\Gamma(a+b)}{\Gamma(a)}U(b,1-a,-z)
%\end{equation}
%where $U$ is the Tricomi confluent hypergeometric function, see e.g.~\cite{Lebedev72}.
From \eqref{Pareto-kth moment} it is easy to see that $\P_{a,b}$ admits moments of order $u$ for any $u \in \left(-b,a \right)$. In particular $\P_{a,b}$ has infinite mean whenever $a \leq 1$.
%Alternative terminology for the variable $\P_{a,b}$ are
%the beta prime distribution, inverted beta distribution or beta distribution of second kind. Moreover it can also be found in the literature under the name generalized gamma-ratio distribution, which is clearly in relation with the distributional equation \eqref{pareto-gamma}.
We refer to \cite{JKB95} for a nice exposition on Pareto variables.
When $\rho=a=1-b$, we write simply  $\P_{\rho}=\P_{\rho,1-\rho}$ for the standard Pareto which  is linked to the generalized arc-sine law $\A_{\rho}  $ of order $\rho$ in the following way
\begin{equation} \label{eq:Pareto-Arc-sine}
(1+\P_{\rho})^{-1} \stackrel{\left(d\right)}{=} \A_{\rho}.
\end{equation}
Simple algebra yields, from the identity \eqref{id-Doney},  the following factorizations of the Pareto law
\begin{equation} \label{id-Doney-pareto and id-Bessel process-Pareto}
	        \frac{\widehat{M}_{\rho}^{\alpha}}{\M_{\rho}^{\alpha}} \stackrel{\left(d\right)}{=} \P_{\rho}.
	\end{equation}
Inspired by this reasoning, we shall prove the factorization of the Pareto variable in terms of exponential functionals of L\'evy processes. To this end, we shall need  the following characterizations of the Mellin transform of the Pareto variable $\P_{\rho}$.
\begin{lemma}\label{lem:par}
For any $0<\rho<1$, we have
\begin{equation}\label{eq:fe}
 \mathcal{M}_{\P_{\rho}}(z+1)=\frac{\Gamma(z+1-\rho)}{\Gamma(1-\rho)}\frac{\Gamma(-z+\rho)}{\Gamma(\rho)}
\end{equation}
which defines an analytical function on the strip $\C_{(\rho-1,\rho)}$  with simple poles at the edges of its domain of analyticity, that is at the points $\rho$ and $\rho-1$.
Moreover, it is the unique positive-definite function solution to the recurrence equation
\begin{equation}\label{eq:fexp}
 \mathcal{M}_{{\P_{\rho}}}(z+1)=- \mathcal{M}_{\P_{\rho}}(z), \quad  \mathcal{M}_{\P_{\rho}}(1)=1.
\end{equation}
\end{lemma}
\begin{proof}
First,
from the definition \eqref{pareto-gamma} of the Pareto variable, one has, for any $0<\rho<1$ and $b>0$,
\begin{equation}\label{eq:fe}
 \mathcal{M}_{\P_{\rho,b}}(z+1)=\mathcal{M}_{\G_ {b}}(z+1)\mathcal{M}_{\G_ {\rho}}(-z+1)=\frac{\Gamma(z+b)}{\Gamma(b)}\frac{\Gamma(-z+\rho)}{\Gamma(\rho)},
\end{equation}
and thus, as $\P_{\rho}=\P_{\rho,1-\rho}$,
\begin{equation}\label{eq:fe}
 \mathcal{M}_{\P_{\rho}}(z+1)=\frac{\Gamma(z+1-\rho)}{\Gamma(1-\rho)}\frac{\Gamma(-z+\rho)}{\Gamma(\rho)}.
\end{equation}
As a by-product of classical  properties of the gamma function, one gets that  $z \mapsto \mathcal{M}_{\P_{\rho}}(z+1)$ defines an analytical function on the strip $\C_{(\rho-1,\rho)}$ and which extends as a meromorphic function on $\C$ with simple poles at the points $\rho+n$ and $\rho-1-n$, $n\in \N$. Then, using the recurrence relation of the gamma function $\Gamma(z+1)=z\Gamma(z), z \in \C$, %\setminus{\mathbb Z_-}$,
we deduce that, for any $z\in \C_{(\rho-1,\rho)}$,
\begin{equation}\label{eq:fe}
 \mathcal{M}_{\P_ {\rho}}(z+1)=\frac{z-\rho}{-z+\rho}\frac{\Gamma(z-\rho)}{\Gamma(1-\rho)}\frac{\Gamma(-z+1+\rho)}{\Gamma(\rho)}=-\mathcal{M}_{\P_ {\rho}}(z),
\end{equation}
which is easily seen to be valid, in fact, for all $z\in \C$.
Next,  we recall that the Stirling's formula yields that for any $a \in \R$ fixed,
\begin{equation}\label{eq:gamma_stir}
\lim_{|b| \rightarrow \infty}|\Gamma(a+ib)| |b|^{-a+\frac{1}{2}} e^{|b|\frac{\pi}{2}} = C_a
\end{equation}
where $C_a > 0$, see e.g.~\cite{Lebedev72}. This combines with \eqref{eq:fe} entails that  for large $|b|$,
\begin{equation*}
|\mathcal{M}_{\P_ {\rho}}(1+i|b|)| \leq \overline{C}_\rho |b|^{-\frac{1}{2}} e^{-|b|}.
\end{equation*}
Since the Mellin transform of a random variable is bounded on the line $1+i\R$,  the uniqueness of the recurrence equation follows readily by an application of Carlson's theorem on the growth of periodic functions, see \cite[Section 5.81]{Titch}.
\end{proof}

We state the following result  which is proved in \cite{Patie-Savov-Bern} regarding the recurrence equation solved by the Mellin transform of the exponential functional  ${\rm{I}}_{\Psi}$ for a general L\'evy process.  Note that the exponential functional is defined in the aforementioned paper with $\widehat{\xi}=-\xi$, that is for $ \widehat{\Psi}(-z)=\Psi(z)$.

\begin{lemma}{\cite[Theorem 2.4]{Patie-Savov-Bern}.} \label{lem:recef}
For any $\Psi \in \mathcal N$, $\mathcal{M}_{{\rm{I}}_{\Psi}}$  is the unique positive-definite function solution to the functional equation
\begin{equation}\label{eq:fexp}
 \mathcal{M}_{{\rm{I}}_{\Psi}}(z+1)=\frac{-z}{\Psi(z)} \mathcal{M}_{{\rm{I}}_{\Psi}}(z), \quad  \mathcal{M}_{{\rm{I}}_{\Psi}}(1)=1,
\end{equation}
which is valid (at least) on the dashed line $\mathcal Z^c_0(\Psi)\setminus \{0\}$, where we set  $\mathcal Z_0(\Psi)=\{z\in i\R;\,\Psi(z)=0\}$. If $\Psi \in \mathcal N_1({\rho})$, $0<\rho<1$, the  validity of the recurrence equation  \eqref{eq:fexp} extends  to $\C_{(0,2)}\cup \mathcal Z^c_0(\Psi) \cup \mathcal Z^c_0(\Psi(.+1))$ and $\mathcal{M}_{{\rm{I}}_{\Psi}}(z+1)$ is analytical on the strip $\C_{(-1,\rho)}$ and meromorphic on $\C_{(-1,1)}$ with $\rho$ as unique simple pole.
\end{lemma}
We mention that in \cite[Theorem 2.1]{Patie-Savov-Bern}, an explicit representation of the solution on a strip of the functional equation \eqref{eq:fexp} is provided in terms of the co-called Bernstein-gamma functions. This representation turns out to be very useful to provide substantial distributional properties, such as a Wiener-Hopf factorization, smoothness, small and large asymptotic behaviors, of the exponential functional of any L\'evy processes, see Section 2 of the aforementioned paper.

%\subsection{End of the proof of Theorem \ref{mainthm}}
We are now ready to complete the proof of Theorem \ref{mainthm}.  Let  $\Psi \in \mathcal N_1({\rho})$ and $\widetilde{\Psi} \in \mathcal N$ and define the random variable  $\overline{\rm{I}}= \frac{{\rm{I}}_{\Psi}}{{\rm{I}}_{\widetilde{\Psi}}}$, where we assumed that the exponential functionals ${\rm{I}}_{\Psi}$ and ${\rm{I}}_{\widetilde{\Psi}}$ are independent variables. Then, plainly
\begin{equation}\label{eq:fer0}
 \mathcal{M}_{\overline{\rm{I}}}(z+1)=\mathcal{M}_{{\rm{I}_\Psi}}(z+1)\mathcal{M}_{{\rm{I}}_{\widetilde\Psi}}(-z+1)
\end{equation}
and,  Lemma \ref{lem:recef} yields, after a shift by $1$,  and with  $\mathcal{M}_{\overline{\rm{I}}}(1)=1$, that
\begin{equation*}%\label{eq:fer}
 \mathcal{M}_{\overline{\rm{I}}}(z+2)=\frac{-z-1}{\Psi(z+1)}\frac{\widetilde{\Psi}(-z)}{z} \mathcal{M}_{\overline{\rm{I}}}(z+1),
\end{equation*}
for (at least) any $z$ on the dashed line  $\mathcal Z^c_0(\Psi(.+1)\cap \mathcal Z^c_0(\widetilde\Psi)\setminus \{0,1\}$.
Therefore, if  one chooses $\widetilde{\Psi}$ of the form
\begin{align*}
\widetilde{\Psi}(-z)& = \frac{z+1}{z} \Psi(z+1),
\end{align*}
that is
$ \widetilde{\Psi}(-z)= \mathcal{T}_{1}\Psi(z)$ or $\widetilde{\Psi}(z)=\widehat{ \mathcal{T}_{1}\Psi}(z)=\widehat{\Psi}_1(z)$, one gets that $\widehat{\Psi}_1 \in \mathcal{N}_1(1-\rho)$ and thus according to Lemma \ref{lem:recef}, $\mathcal{M}_{{\rm{I}}_{\widehat{\Psi}_1}}(-z+1)$ is analytical on the strip $\C_{(1-\rho,1)}$ with $1-\rho$ as a simple pole.  Then we obtain, from \eqref{eq:fer0}, that $ \mathcal{M}_{\overline{\rm{I}}}(z+1)$ is analytical on the strip $\C_{(1-\rho,\rho)}$ with $1-\rho$ and $\rho$ as simple poles and it is solution to the recurrence equation
\begin{equation*}%\label{eq:fer}
 \mathcal{M}_{\overline{\rm{I}}}(z+1)=-\mathcal{M}_{\overline{\rm{I}}}(z).
\end{equation*}
Since plainly $\mathcal{M}_{\overline{\rm{I}}}$ is a positive-definite function we conclude   by the uniqueness claimed in Lemma \ref{lem:par} that
\begin{equation} \label{eq:Par}
\frac{{\rm{I}}_{{\Psi}}}{{\rm{I}}_{{\widehat{\Psi}_{1}}}} \stackrel{\left(d\right)}{=} \P_{\rho}.
\end{equation}
Invoking the identity \eqref{eq:Pareto-Arc-sine} completes the proof.

\subsection{Proof of Corollary \ref{cor:par}}
First, from \eqref{eq:Par}, we deduce easily that  $\frac{{\rm{I}}_{{\widehat{\Psi}_{1}}}}{{\rm{I}}_{{\Psi}}}$ has the same law as $\P_{1-\rho}$.  The fact that the density of the Pareto variable is hyperbolic completely monotone was proved in \cite{Bond}.
Moreover, Berg \cite{Berg} showed that $ \log \G_a$ is infinitely divisible for any $a>0$, we conclude the proof by recalling that the set of infinitely divisible variables  is closed by linear combination of independent variables.  Finally, the last claim follows readily from the definition of $\P_{\rho}$ and the connection with the standard Cauchy variable, which was observed  by Pitman and Yor in \cite{PY}.  %The fact that the variable $\P_ {a,b}$ is infinitely divisible was first proved by  They show that $\P_ {a,b}$ is a GGC variable and hence infinitely divisible. An independent proof of the infinite divisibility property of $\P_ {a,b}$ has also been given by Ismail and Kelker in \cite{IK79}, where they show that the distribution of $\P_ {a,b}$ is self-decomposable and hence infinitely divisible.
%The GGC property of the variable $\P_ {a,b}$ can also been found in later papers for densities generalizing \eqref{Pareto-density} as e.g in \cite{B79, B81}.
%We point out that in  \cite[section 2.5]{JRY08}, James et al.~compute the Thorin measure associated with $\P_ {a,b}$ for any $a>0$ and $0<b<1$.

\subsection{Proof of Corollary \ref{cor:sr}}
In order to prove the identity \eqref{id-Doney},  we first recall the connection between the law of the maximum of a stable process and the exponential functional of a specific  L\'evy process, usually referred to as the Lamperti-stable process. This link has been established through the so-called Lamperti transform and we refer to \cite{CC}, \cite{KP13} and  \cite[Section 2.2]{Kypri} for more details. We proceed by providing the L\'evy-Khintchine   exponent  $\Psi_{\alpha,\rho}$ of the Lamperti-stable process of parameters $(\alpha,\rho),  \alpha \in (0, 2)$ and $\rho \in (0,1)$,  which is given by
\begin{equation} \label{eq: LE examples A Doney}
\Psi_{\alpha,\rho}(z) = - \frac{\Gamma(1+\alpha z)}{\Gamma(1-\alpha \rho+\alpha z)} \frac{\Gamma(\alpha-\alpha z)}{\Gamma(\alpha \rho-\alpha z)}, \:  z \in \C_{(-\frac{1}{\alpha},1)},
\end{equation}
see \cite[Theorem 2.3]{Kypri} where we consider here the exponent of $\alpha \xi^*$ in the notation of that paper.
The following identity in law between the suprema of stable processes and the exponential functional of L\'evy processes can be found, for example,  in \cite[p.~133]{KP13},
\begin{equation} \label{eq: IDiL EF Sub stable}
	         \M_{\rho}^{-\alpha} \stackrel{\left(d\right)}{=}{\rm{I}}_{\Psi_{\alpha,\rho}}
	      \end{equation}
where we recall that $\M_{\rho}=\sup_{0\leq t \leq 1}X_{t} $ and $X=\left(X_{t}\right)_{t\geq0}$ is an $\alpha$ stable process with positivity parameter $\rho$. Next, observe  that $\Psi_{\alpha,\rho}(\rho)=0$ and using the recurrence relation of the gamma function, easy algebra yields
 \begin{eqnarray*}
  \frac{z}{z+1}\Psi_{\alpha,\rho}(z+1)&=& -\frac{z}{z+1}\frac{\Gamma(1+\alpha +\alpha z)}{\Gamma(1+\alpha(1- \rho)+\alpha z)} \frac{\Gamma( -\alpha z)}{\Gamma(-\alpha(1- \rho)-\alpha z)}\\
  &=&-\frac{\Gamma(\alpha +\alpha z)}{\Gamma(\alpha(1- \rho)+\alpha z)} \frac{\Gamma(1-\alpha z)}{\Gamma(1-\alpha(1- \rho)-\alpha z)}.
\end{eqnarray*}
Then, we get that $\lim_{u\to 0}\frac{u}{u+1}\Psi_{\alpha,\rho}(u+1)=-\frac{\Gamma(\alpha)}{\Gamma(\alpha(1- \rho))} \frac{1}{\Gamma(1-\alpha(1- \rho))}\leq 0$ as always $\alpha(1- \rho)\leq 1$.
We could easily check from the form  \eqref{eq: LE examples A Doney} of $\Psi_{\alpha,\rho}$ and the expression of its L\'evy measure given in \cite{KP13} that $y\mapsto e^{y} \overline{\Pi}_{+}\left(y\right)$ is non-decreasing on $\R^+$. Instead, we simply observe from the computation above that
\begin{equation*}
 \widehat{\Psi}_{1}(z) = \mathcal{T}_{1} \Psi_{\alpha,\rho} (-z) = 	 - \frac{\Gamma(\alpha -\alpha z)}{\Gamma(\alpha(1- \rho)-\alpha z)} \frac{\Gamma(1+\alpha z)}{\Gamma(1-\alpha(1- \rho)+\alpha z)}=\Psi_{\alpha,1-\rho}(z),
\end{equation*}
which is the characteristic exponent of the Lamperti-stable process with parameter $(\alpha, 1-\rho)$.
Hence,  similarly to \eqref{eq: IDiL EF Sub stable}, we have the following identity in law
\begin{equation*}
	         \widehat{\M}_{1}^{-\alpha} \stackrel{\left(d\right)}{=} {\rm{I}}_{\widehat{\Psi}_{1}},
	      \end{equation*}
which by an application of Theorem \ref{mainthm} completes the proof.

\subsection{Proof of Corollary \ref{cor:s2}}
Let us now consider, for any $\alpha \in (0,1)$,
 \begin{equation*}% \label{ex:LE}
\widehat\Psi_{\alpha}(z) = \frac{\Gamma(1+\alpha-\alpha z)}{\alpha \Gamma(-\alpha z)}, \ z\in  \C_{(-\infty,1+\frac{1}{\alpha})}.
\end{equation*}
In \cite[Section 3.1]{P11}, it is shown that $\widehat\Psi_{\alpha}$ is  the L\'evy-Khintchine exponent of a spectrally positive L\'evy process with a negative mean  and that ${\rm{I}}_{\widehat\Psi_{\alpha}}$ is a positive self-decomposable variable with
\[{\rm{I}}_{\widehat\Psi_{\alpha}}\stackrel{(d)}{=} \e^{-\alpha}. \]
 In other words, ${\widehat{\rm{I}}}_{\Psi_{\alpha}}$  has the Fr\'echet distribution of parameter $\rho=\frac1\alpha>1$.
 Observe that  $\widehat\Psi_{\alpha}(\rho)=0$ and thus $\widehat\Psi_{\alpha}$ does not satisfy the hypothesis of Theorem \ref{mainthm}. However, in \cite[Section 3.1]{P11} it is also shown that, up to a positive multiplicative constant, the tail of the L\'evy measure of $\widehat\Psi_{\alpha}$ is given by $\overline{\Pi}_{\alpha}(y)=e^{-(\alpha+1)y/\alpha}(1 - e^{-
y/\alpha})^{-\alpha-1}, y>0,$ and thus plainly the mapping  $e^{\beta y}\overline{\Pi}_{\alpha}(y)$ is non-decreasing on $\R^+$ for any $\beta \leq \frac1\alpha+1 $. Then, Lemma \ref{lem:par} gives, for any $\rho=\beta -\frac1\alpha \leq 1 $, that
\begin{equation*}% \label{ex:LE3}
\widehat\Psi_{\alpha,\rho}(z)=\mathcal T_{\rho +\frac1\alpha} \widehat\Psi_{\alpha}(z)= \frac{z}{z+\rho +\frac1\alpha}  \frac{\Gamma(\alpha-\alpha\rho -\alpha z)}{\alpha\Gamma(-1-\alpha\rho -\alpha z)} =- z  \frac{\Gamma(\alpha-\alpha\rho -\alpha z)}{\Gamma(-\alpha\rho -\alpha z)}
\end{equation*}
is, since $\lim_{u \to 0}\widehat\Psi_{\alpha,1}(u)=\lim_{u \to 0} \frac{\Gamma(1 -\alpha u)}{\alpha\Gamma(-\alpha -\alpha u)} =-\frac{1}{\Gamma(1-\alpha)}<0$,  the characteristic exponent of a L\'evy process, as well as, by duality,
\[  \Psi_{\alpha,\rho}(z) =  z  \frac{\Gamma(\alpha-\alpha\rho +\alpha z)}{\Gamma(-\alpha\rho +\alpha z)}. \]
Note that, if $0<\rho< 1$ then $\Psi_{\alpha,\rho}(\rho)=0$ and we deduce by convexity and since $ \Psi_{\alpha,\rho}(0)=0$ that $ \Psi_{\alpha,\rho}'(0^+)<0$. Hence $ \Psi_{\alpha,\rho} \in \mathcal{N}_1(\rho)$.   Using Lemma \ref{lem:recef}, one gets that
\begin{equation}\label{eq:fexpex}
 \mathcal{M}_{{\rm{I}}_{ \Psi_{\alpha,\rho}}}(z+1)=-\frac{\Gamma(\alpha-\alpha\rho +\alpha z)}{\Gamma(-\alpha\rho +\alpha z)} \mathcal{M}_{{\rm{I}}_{ \Psi_{\alpha,\rho}}}(z), \quad  \mathcal{M}_{{\rm{I}}_{\Psi_{\alpha,\rho}}}(1)=1.
\end{equation}
As mentioned earlier, \cite[Theorem 2.1]{Patie-Savov-Bern} provides the solution of this functional equation which is derived as follows. First we recall that the analytical Wiener-Hopf factorization of $ \Psi_{\alpha,\rho}$ is given by
\begin{equation*}%\label{eq:fexp}
 \Psi_{\alpha,\rho}(z)=-(-z+\rho)\phi_{\rho}(z)
\end{equation*}
where $\phi_{\rho}(z)=  \alpha z  \frac{\Gamma(\alpha-\alpha\rho +\alpha z)}{\Gamma(1-\alpha\rho +\alpha z)}$ is a Bernstein function, see \cite{Dbook}. Then,  the solution of \eqref{eq:fexpex}  takes the form
\begin{equation*}%\label{eq:fexpex}
 \mathcal{M}_{{\rm{I}}_{ \Psi_{\alpha,\rho}}}(z+1)=\frac{\Gamma(z+1)}{W_{\phi_{\rho}}(z+1)}\Gamma(\rho - z),
\end{equation*}
where $W_{\phi_{\rho}}(z+1)= \phi_{\rho}(z)W_{\phi_{\rho}}(z),\: W_{\phi_{\rho}}(1)=1.$ To solve this latter recurrence equation, we note that $\phi_{\rho}(z)=  \alpha  \frac{z}{z-\rho}  \frac{\Gamma(\alpha(z+1-\rho))}{\Gamma(\alpha(z-\rho))}$, then easy algebra and the uniqueness argument used in the proof of Lemma \ref{lem:par} yield that $W_{\phi_{\rho}}(z+1)= \alpha^z \frac{\Gamma(1-\rho)\Gamma(z+1)\Gamma(\alpha(z+1-\rho))}{\Gamma(z+1-\rho)\Gamma(\alpha(1-\rho))}$ and thus
\begin{eqnarray*}%\label{eq:fexpex}
 \mathcal{M}_{{\rm{I}}_{ \Psi_{\alpha,\rho}}}(z+1)&=&\alpha^{-z} \frac{\Gamma(z+1-\rho)\Gamma(\alpha(1-\rho))}{\Gamma(1-\rho)\Gamma(z+1)\Gamma(\alpha(z+1-\rho))}\Gamma(z+1)\Gamma(\rho - z) \\
 &=&\alpha^{-z} \frac{\Gamma(z+1-\rho))}{\Gamma(\alpha(z+1-\rho))}\frac{\Gamma(\alpha(1-\rho))}{\Gamma(1-\rho)}\Gamma(\rho - z).
\end{eqnarray*}
Next, recalling that $S_{\gamma}(\alpha)$ is the $\gamma$-length biased variable of a positive $\alpha$-stable random variable, we observe, from \cite[Section 3(3)]{P11} that
  \begin{equation*}
\E\left[ S_{1-\rho}^{-\alpha z}(\alpha)\right] = \frac{\E\left[S^{-\alpha (1-\rho+z)}(\alpha) \right]}{ \E\left[S^{-\alpha(1-\rho)}(\alpha)\right]}= \frac{\Gamma(z+1-\rho))}{\Gamma(\alpha(z+1-\rho))}\frac{\Gamma(\alpha(1-\rho))}{\Gamma(1-\rho)}.
\end{equation*}
Thus, by Mellin transform identification, we get that
\begin{equation*}
	          {\rm{I}}_{{\Psi}_{\alpha,\rho}} \stackrel{\left(d\right)}{=} \alpha^{-1} S_{1-\rho}^{-\alpha }(\alpha) \times \G^{-1}_{1-\rho}
	      \end{equation*}
where the variables on the right-hand side are taken independent.
Next, we have, writing simply $\Psi_{1}=\mathcal T_1\Psi_{\alpha,\rho}$,
\[
\mathcal T_1\Psi_{\alpha,\rho}(z) =  z  \frac{\Gamma(\alpha(2-\rho +z))}{\Gamma(\alpha(1-\rho+ z))} \]
which yields
\begin{equation*}%\label{eq:fexpex}
 \mathcal{M}_{{\rm{I}}_{ \widehat \Psi_{1}}}(z+1)=\frac{\Gamma(\alpha(1-\rho- z))}{\Gamma(\alpha(2-\rho -z))}  \mathcal{M}_{{\rm{I}}_{\widehat \Psi_{1}}}(z), \quad  \mathcal{M}_{\widehat \Psi_{1}}(1)=1.
\end{equation*}
 It is not difficult to check that $\mathcal{M}_{{\rm{I}}_{ \widehat \Psi_{1}}}(z+1)=\Gamma(\alpha(1-\rho- z))$ is the unique positive-definite solution of this equation. Invoking Theorem \ref{mainthm} completes the proof.

\bibliographystyle{abbrv}

%\nocite{*}

	%\bibliography{cabib}

\end{document}